\documentclass{article}
\usepackage[utf8]{inputenc}
\usepackage{amsmath}
\usepackage{amssymb}
\usepackage{amsfonts}
\usepackage{amsthm}
\usepackage{mathrsfs}
\usepackage{bbm}
\usepackage{comment}
\usepackage{enumerate}
\usepackage{graphicx}
\usepackage[dvipsnames]{xcolor}
\usepackage{subcaption}
\usepackage{float}
\usepackage{hyperref}
\hypersetup{
    colorlinks=true,
    linkcolor=Violet,
    citecolor=Violet,
    urlcolor=Violet,
}
\usepackage{cite}

\theoremstyle{plain}
\newtheorem{theorem}{Theorem}
\newtheorem{lemma}[theorem]{Lemma}
\newtheorem{proposition}[theorem]{Proposition}
\newtheorem{proposition*}{Proposition}
\newtheorem{corollary}[theorem]{Corollary}

\theoremstyle{definition}

\newtheorem{example}{Example}

\theoremstyle{remark}

\numberwithin{equation}{section}

\newcommand{\E}{\mathbb{E}}
\newcommand{\F}{\mathscr{F}}
\newcommand{\p}{\mathbb{P}}
\DeclareMathOperator{\Var}{\text{Var}}
\DeclareMathOperator{\N}{\mathbb{N}}

\DeclareMathOperator{\Z}{\mathbb{Z}}
\DeclareMathOperator{\X}{\mathcal{X}}
\DeclareMathOperator{\A}{\mathscr{A}}

\makeatletter
\newcommand{\ltext}[2]{%
  \@bsphack
  \csname phantomsection\endcsname 
  \def\@currentlabel{#1}{\label{#2}}%
  \@esphack
}

\newcommand{\subjclass}[2][1991]{%
  \let\@oldtitle\@title%
  \gdef\@title{\@oldtitle\footnotetext{#1 \emph{Mathematics subject classification.} #2}}%
}
\newcommand{\keywords}[1]{%
  \let\@@oldtitle\@title%
  \gdef\@title{\@@oldtitle\footnotetext{\emph{Key words and phrases.} #1.}}%
}
\makeatother

\title{Limit theorems for a random walk with memory perturbed by a dynamical system}

\keywords{Elephant Random Walk, Dynamical System, Strong Law of Large Numbers, Central Limit Theorem}

\author{Cristian F. Coletti
    \footnote{Centro de Matem\'atica, Computa\c{c}\~ao e Cogni\c{c}\~ao, Universidade Federal do ABC, Av. dos Estados, 5001, 09210-580 Santo Andr\'e, S\~ao Paulo, Brazil (\url{cristian.coletti@ufabc.edu.br}, \url{lucas.roberto@ufabc.edu.br}, \url{denis.luiz@ufabc.edu.br}).}
    \and Lucas R. de Lima
    \footnotemark[1]
    \and Renato J. Gava \footnote{Departamento de Estat\'istica, Universidade Federal de S\~ao Carlos, Rod. Washington Luiz, km 235, S\~ao Carlos, S\~ao Paulo, 13565-905, Brazil (\url{ gava@ufscar.br})}
    \and  Denis A. Luiz
    \footnotemark[1]}
\date{}
\begin{document}

\nocite{*}

\maketitle

\begin{abstract}
We introduce a new random walk with unbounded memory obtained as a mixture of the Elephant Random Walk and the Dynamic Random Walk which we call the Dynamic Elephant Random Walk (DERW). As a consequence of this mixture the distribution of the increments of the resulting random process is time dependent. We prove a strong law of large numbers for the DERW and, in a particular case, we provide an explicit expression for its speed. Finally, we give sufficient conditions for the central limit theorem and the law of the iterated logarithm to hold.
\end{abstract}


\section{Introduction}

In this work we introduce a new random walk with memory. The model is obtained as a mixture of two well known random walks, namely the Elephant Random Walk and the Dynamic Random Walk and it is inspired from the theory of Markov switching models where switching among regimes occurs randomly according to a Markov process, see Hamilton \cite{hamilton1989new}. The mixture of models has already been considered in the literature of interacting particle systems. For instance the mixture of spin-flip dynamics and symmetric exclusion processes are known under the name of diffusion-reaction processes, see  Belitsky et al. \cite{belitsky2001} and references there in, and have been intensively studied.

Recently, the Elephant Random Walk (ERW) and other random walks with memory have received considerable attention by many authors, see \cite{baur2020,baur2016,benari2019,bercu2017,bercu2019center,bercu2019multi,bercu2019hypergeometric,bertoin2020universality,dasilva2020,dimolfetta2018,kubota2019,marquioni2019,vazquez-guevara2019}. The ERW was introduced by  Sch\"{u}tz and Trimper \cite{schuetz2004} and its dynamic is defined as follows. The elephant starts at the origin and moves one step to the right with probability $q$ and one step to the left with
probability $1-q$. At time $n+1$, it chooses at random and with equal probability a number $n^{\prime}$ from the set $\{1, \ldots , n\}$. Then the walker takes, with probability $p$, one step in the same direction of the step given at time $n^{\prime}$ and, with probability $1-p$, it takes one step in the opposite direction. See section \ref{themodel} for a formal definition of the ERW.

The Dynamic Random Walk (DRW) is a non-homogeneous Markov chain which was introduced by  Guillotin-Plantard \cite{guillotin-plantard2000} and whose transition probability of each step is time dependent. In this paper we consider DRW on $\mathbb{Z}$ which evolves in the following manner: at each time a walker takes one step to the right or one step to the left with probability given by a function of the orbit of a given discrete-time dynamical system. See section \ref{themodel} for a formal definition of the DRW.

This paper is organized as follows. In section \ref{themodel} we introduce the Dynamic Elephant Random Walk. In section \ref{s3} we state our main results and we exhibit some examples where these results hold. In section \ref{s4} we prove the strong law of large numbers, the central limit theorem and the law of the iterated logarithm for the Dynamic Elephant Random Walk.

\section{The model}\label{themodel}
We begin this section by defining the Elephant Random Walk (ERW) $(V_n)_{n \geq 0}$ with $V_0 = 0$ and increments $(W_n)_{n \geq 1}$. Let $p, q \in [0,1]$, then

\begin{align*}
    &\p^E[W_1=1]=q,\quad \quad \p^E[W_1=-1]=1-q, 
    \end{align*}
    and
    \begin{align*}
    &\p^E[W_n=\eta|W_1,\dots,W_{n-1}]=\frac{1}{2n}\sum_{k=1}^{n-1}(1+(2p-1)W_k\eta).
\end{align*}
The random variables $(V_n)_{n \geq 0}$ and $(W_n)_{n \geq 1}$ are related by the formula $V_n = \sum_{i=1}^n W_i$ where $n \geq 1$. Denote by $p^E_n (.)$ the probability mass function governing the law of the increments of the ERW.

Then we define the Dynamic Random Walk (DRW). Let $(\X,\A,\mu,T)$ be a dynamical system where $(\X,\A,\mu)$ is a probability space and $T:\X\to\X$ is a $\mu$-invariant transformation, \emph{i.e.}, $\mu(A)=\mu\big(T(A)\big)$ for all $A\in\A$. The DRW is the stochastic process $(Y_n)_{n \geq 0}$ with $Y_0=0$ and increments $(Z_n)_{n \geq 1}$ defined by
\begin{equation*}
    \p_x^D[Z_n=1]=f(T^nx)
    \quad\text{and}\quad
    \p_x^D[Z_n=-1]=1-f(T^nx)
\end{equation*}
where $f:\X\to[0,1]$ is $\A$-measurable. Here, $T^0=id$ and, for $n \geq 1, T^{n}:=T\circ T^{n-1}$. The random variables $(Y_n)_{n \geq 0}$ and $(Z_i)_{i \geq 1}$ are related by the formula $Y_n = \sum_{i=1}^n Z_i$ where $n \geq 1$. Denote by $p^D_{x,n} (.)$ the probability mass function governing the law of the increments of the DRW

Now we introduce the Dynamic Elephant Random Walk (DERW) on $\Z$. Let $(\mathcal{X},\A, \mu,T)$ and $f:\X\to[0,1]$ be as in the definition of the DRW. Fix $g:R \subset \mathbb{R} \times \mathbb{N} \to[0,1]$ and $p,q \in [0,1]$. The DERW is the random walk $(S_n)_{n \geq 0}$ with $S_0 = 0$ and increments $(X_n)_{n \geq 1}$ defined by
\begin{equation*}
    \p_x[X_n=\eta]=g(\alpha,n) p^E_n(\eta)+\big(1-g(\alpha,n)\big)p^D_{x,n}(\eta).
\end{equation*}
Here $S_n = \sum_{i=1}^n X_i$. We will denote by $\E_x$ the operator expectation induced by $\p_x$.

\section{Main results}\label{s3}
Let $(S_n)_{n \geq 0}$ be the DERW with increments $(X_n)_{n\geq1}$. Set $\F_n:=\sigma\langle X_1,\dots,X_n\rangle$ with $\F_0=\{\emptyset,\X\}$. We will exclude the case where $g(\alpha,2)=1$ and $p=0$ without further notice, which is technically inconvenient for our results. To avoid cumbersome notation, we will write for short $\alpha_n:=g(\alpha,n)$ for a given $\alpha \in \mathbb{R}$. Set
\begin{eqnarray*}
    \ell_{\inf}(\alpha):= \liminf\limits_{n\to +\infty}g(\alpha,n), \quad \ell_{\sup}(\alpha):= \limsup\limits_{n\to +\infty}g(\alpha,n),\\
    \text{and} \quad \ell(\alpha):= \lim\limits_{n\to +\infty}g(\alpha,n), \;\text{when it exists.}
\end{eqnarray*}

\noindent\textbf{Condition (T).} It is convenient to consider the DERW satisfying the following condition:
\begin{equation*}
    p\cdot \ell_{\sup}(\alpha) <1, \ltext{T}{transition:property}
\end{equation*}
which will be referred throughout the text as condition \eqref{transition:property}.

\begin{theorem}[Strong Law of Large Numbers]\label{SLLN}
If the DERW satisfies condition \eqref{transition:property}, then
\begin{equation*}
    \frac{S_n - \E_x[S_{n}]}{n} \to 0 \quad \p_x-a.s.
\end{equation*}
for every $x \in \X$.
\end{theorem}

It is worth remarking that if $p=1/2$ and the limit $\ell (\alpha )$ exists then we may compute the speed of the DERW.

\begin{corollary}\label{limitp2}
Let $(S_n)_{n \geq 0}$ be the DERW. If $p=1/2$ and $\ell(\alpha)>0$ then, for $\mu$-almost every $x \in \X$, 
\begin{equation*}
    \lim_{n\to+\infty}\frac{S_n}{n}=(1-\ell(\alpha))(2\E_{\mu}[f|\mathcal{I}](x)-1)\quad\p_x-a.s. .
\end{equation*}
where $\mathcal{I}$ stands for the $\sigma$-algebra of $T$-invariant sets.
\end{corollary}

Corollary \ref{limitp2} provides sufficient conditions for the existence of $\lim\limits_{n\to+\infty}\frac{S_n}{n}$. In the following example we provide a method to compute explicitly the speed of the DERW under some mild conditions.

\begin{example}
Consider a DERW satisfying condition \eqref{transition:property}. In view of Birkhoff's Ergodic Theorem, let $\mathcal{E}\subset\X$ be a set of full measure such that for every $x\in \mathcal{E}$,
\begin{equation} \label{Birkhoff:eq}
    \lim_{n\to+\infty}\frac{1}{n}\sum_{k=0}^{n-1}f(T^kx)=\E_\mu[f|\mathcal{I}](x).
\end{equation} 
Assume that, for a given  $x\in \mathcal{E}$, $\ell(\alpha)$, $\lim\limits_{n\to+\infty}f(T^nx)$ and $\lim\limits_{n\to+\infty}\frac{S_n}{n}$ ($\p_x-a.s.$) do exist. From Cesàro Mean Convergence Theorem we have that
\begin{equation*}
\lim_{n\to+\infty}\frac{1}{n}\sum_{k=1}^nf(T^kx)=\lim_{n\to+\infty}f(T^nx).
\end{equation*}
Denote by $S_x$ the $\p_x$-a.s. limit of $S_n/n$.
Then, invoking Theorem \ref{SLLN} we get $\lim\limits_{n\to+\infty}\frac{\E_x[S_n]}{n}=S_x\;\;\p_x-a.s.$ It is easy to verify that
\begin{eqnarray*}
\lim_{n\to+\infty}\E_x[X_n] = (2p-1)\ell(\alpha)S_x+(1-\ell(\alpha))(2\E_{\mu}[f|\mathcal{I}](x)-1),
\end{eqnarray*}
see equation \eqref{conditional}. Invoking again Cesàro Mean Convergence Theorem we get
\begin{equation*}
\lim_{n\to+\infty}\frac{\E_x[S_n]}{n}=(2p-1)\ell(\alpha)S_x+(1-\ell(\alpha))(2\E_{\mu}[f|\mathcal{I}](x)-1).
\end{equation*}
Therefore
\begin{equation*}
S_x=\frac{1-\ell(\alpha)}{1-(2p-1)\ell(\alpha)}(2\E_{\mu}[f|\mathcal{I}](x)-1)
\end{equation*}
which agrees with the conclusion of Corollary \ref{limitp2} when $p=1/2$.
\end{example}

In what follows, we assume that the DERW satisfies \eqref{transition:property}. Let
\[
a_n:=\prod_{j=1}^{n-1}\left(1+\frac{(2p-1)g(\alpha,j+1)}{j}\right),
\]
and set $(A_n)_{n \geq 1}$ to be the sequence of non-negative constants defined by
\[
A_n^2:=\sum_{k=1}^n\frac{1}{a_k^2}(1-\E_x[X_k]^2).
\]

We state below a version of the Central Limit Theorem for the DERW followed by some variations of the same result. 

\begin{theorem}[Central Limit Theorem]\label{CLT1}
Let $(S_n)_{n \geq 0}$ be the DERW with $p \leq 3/4$ or $\ell_{\sup}(\alpha)<\frac{1}{4p-2}$. For all $x \in \X$, one has that if $p<1$ and $\ell_{\inf}(\alpha) >0$, then
\begin{equation*}
    \frac{S_n-\E_x[S_n]}{a_nA_n}\xrightarrow{\mathcal{D}}\mathcal{N}(0,1).
\end{equation*}
\end{theorem}

\begin{corollary} \label{cor:D}
Let the hypothesis of $p<1$ and $\ell_{\inf}(\alpha) >0$ in Theorem \ref{CLT1} be replaced by condition \eqref{transition:property} jointly with
\begin{enumerate}[(D$_1$)]
    \item \ltext{\normalfont D$_1$}{D1} If $p=1$, then $\liminf\limits_{n\to+\infty} f(T^n(x))>0$ and $\limsup\limits_{n\to+\infty} f(T^n(x))<\frac{1-\ell_{\sup}(\alpha)}{1-\ell_{\inf}(\alpha)}$;
    and
    \item \ltext{\normalfont D$_2$}{D2} If $\ell_{\inf}(\alpha) =0$, then
    $\liminf\limits_{n\to+\infty} f(T^n(x))>0$, $\limsup\limits_{n\to+\infty} f(T^n(x))<
            1-p\cdot\ell_{\sup}(\alpha)$, and $\ell_{\sup}(\alpha) <1$.
\end{enumerate}

\noindent Then the conclusion of Theorem \ref{CLT1} remains true.
\end{corollary}

\begin{corollary}\label{CLTp2}
Let $(S_n)_{n \geq 0}$ be the DERW. If $p=1/2$ and $\ell(\alpha)>0$, then, for $\mu$-almost surely $x \in \X$,
\begin{equation*}
    \frac{S_n-n(1-\ell(\alpha))(2\E_\mu[f|\mathcal{I}](x)-1)}{\sqrt{n\big(1-(1-\ell(\alpha))^2(4\E_\mu[f^2|\mathcal{I}](x)-4\E_\mu[f|\mathcal{I}](x)+1)\big)}}\xrightarrow{\mathcal{D}}\mathcal{N}(0,1).
\end{equation*}
\end{corollary}

\begin{theorem}[Law of Iterated Logarithm] \label{LIL}
Under the conditions of Theorem \ref{CLT1} or Corollary \ref{cor:D},
\begin{align}\label{log}
    \limsup_{n \to \infty}  \dfrac{|S_n - \E_{x}[S_n]|}{a_n A_n\sqrt{\log \log(A_n)}} = \sqrt{2} \quad \mathbb{P}_x-a.s. 
\end{align}
\end{theorem}

We also obtain an almost sure convergence result for the DERW in the regime where the central limit theorem does not hold.

\begin{theorem}\label{ASC}
Let $(S_n)_{n \geq 0}$ be the DERW with $p>3/4$ and $\ell_{\inf}(\alpha)>\frac{1}{4p-2}$. Then, for all $x \in \X$, one has that
\begin{equation*}
    \frac{S_n-\E_x[S_n]}{a_n}\to M\quad \p_x-a.s.,
\end{equation*}
where $M$ is a non-degenerated zero mean random variable.
\end{theorem}
We finish this section with an example of a DERW where the central limit theorem holds.

\begin{example}
Let $\X=\{(x,y)\in\mathbb{R}^2:x^2+y^2\leq1\}$ be endowed with the Lebesgue $\sigma$-algebra $\A$. Denote by $\mu$ the uniform probability measure defined on $\A_{|S^2}$ where $S^2:=\{(x,y)\in\mathbb{R}^2 : x^2+y^2 = 1 \}$. 
Namely, $\mu(I(t))=\frac{t}{2\pi}\wedge1$ for $t\geq0$ where $I(t):=\{(\cos{\beta},\sin{\beta}):\beta\in[0,t)\}$. 

Consider the following system of ordinary differential equations on $\X$

\begin{equation} \label{ode}
\left\{\begin{array}{l}
\frac{dx}{dt}=-y+bx(x^2+y^2-1)\\
\frac{dy}{dt}=x+by(x^2+y^2-1)
\end{array}\right.
\end{equation}

Assume that $b < 0$ and, for $(x,y)\neq(0,0)$, denote by $\phi_t(x,y):=\phi(t,x,y)$ the solutions of \eqref{ode} satisfying $\phi_0=Id$. It is easy to verify, by passing to polar coordinates if necessary, that $\phi_t:\X\to\X$ is $\mu$-invariant for any $t\geq0$.

\begin{figure}[H]
    \centering
    \includegraphics[scale=0.7]{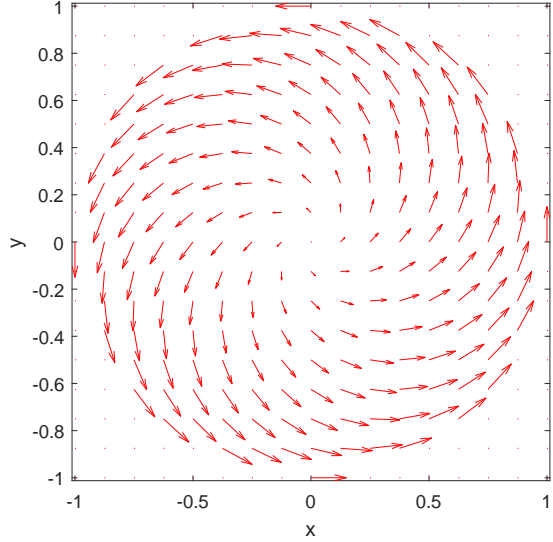}
    \caption{Phase portrait of $\phi_t(x,y)$ for $b=-1$}\label{dynamicmap}
\end{figure}

We observe that for any $m,n \geq 0$, $\phi_n\circ\phi_m=\phi_{n+m}$ (see \cite[p.~175]{hirsch1974differential}). Set $T:=\phi_1$. In polar coordinates we have that
\begin{equation*}
T(r,\theta)=\left\{\begin{array}{ll}
\left(\frac{e^b}{\sqrt{e^{2b}+\frac{1}{r^2}-1}},\theta+1\right),&\text{if }0<r<1\nonumber\\ 
(1,\theta+1),&\text{if }r=1.
\end{array}\right.
\end{equation*}
Let $f:\X\to[0,1]$ be given in polar coordinates by $f(r,\theta)=\cos^2{\theta}$. It follows from Birkhoff's Ergodic Theorem that
\begin{equation}
\E_\mu[f|\mathcal{I}]=\lim_{n\to+\infty}\frac{1}{n}\left(\sum_{k=1}^n\cos^2{(\theta+k)}\right)\mathbbm{1}_{S^2}=\frac{1}{2}\mathbbm{1}_{S^2} \quad \mu-a.s.\nonumber
\end{equation}
where the $\sigma$-algebra of all $T$-invariant sets $\mathcal{I}$ coincides with $\A$. If $g:[0,1]\times\N\to[0,1]$ is given by
\[
    g(\alpha,n)=\frac{\alpha}{2}+(-1)^{n+1}\frac{\alpha}{n+1},
\]
then $\ell(\alpha)=\frac{\alpha}{2}$. Assuming that $p=1/2$ we obtain from Corollary \ref{limitp2} that
\begin{equation}
\lim_{n\to+\infty}\frac{S_n}{n}=\left(1-\frac{\alpha}{2}\right)(\mathbbm{1}_{S^2}-1) \quad \p_{(x,y)}-a.s.\nonumber
\end{equation}
for every $(x,y)\in \X$. A straightforward computation yields
\begin{equation}
\E_\mu[f^2|\mathcal{I}]=\frac{3}{8}\mathbbm{1}_{S^2} \quad \mu-a.s.\nonumber
\end{equation}
Thus we may conclude by means of Corollary \ref{CLTp2} that
    \begin{equation} \label{disc:CLT:example}
        \frac{S_n-n(1-\frac{\alpha}{2})(\mathbbm{1}_{S^2}-1)}{\sqrt{n\big(1+(1-\frac{\alpha}{2})^2(1-\frac{1}{2}\mathbbm{1}_{S^2})\big)}}\xrightarrow{\mathcal{D}}\mathcal{N}(0,1)
    \end{equation}
    for all $(x,y) \in \mathcal{E}$, where $\mathcal{E}$ is as in \eqref{Birkhoff:eq}.
Hence, if $(x,y)\in \mathcal{E}\setminus S^2$, then
    \begin{equation*}
        \frac{S_n+n(1-\frac{\alpha}{2})}{\sqrt{n}}\xrightarrow{\mathcal{D}}\mathcal{N}\left(0,1+\left(1-\frac{\alpha}{2}\right)^2\right),
    \end{equation*}
Moreover, if $(x,y) \in \mathcal{E} \cap S^2$, then
    \begin{equation*}
        \frac{S_n}{\sqrt{n}}\xrightarrow{\mathcal{D}}\mathcal{N}\left(0,1+\frac{1}{2}\left(1-\frac{\alpha}{2}\right)^2\right).
    \end{equation*}
\end{example}
\section{Proofs}\label{s4}
\subsection{The Strong Law of Large Numbers}
The following equation can be obtained by straightforward computation. It gives us an explicit expression for the conditional expectation with respect to the natural filtration of the increments of the Dynamic Elephant Random Walk.
\begin{equation}\label{conditional}
    \E_x[X_{n+1}|\F_n]=\frac{\alpha_{n+1}(2p-1)}{n}S_n+(1-\alpha_{n+1})(2f(T^{n+1}x)-1)
\end{equation}

Before proving the Strong Law of Large Numbers, we state the lemma below.

\begin{lemma}\label{ann}
Assume that the DERW satisfies \eqref{transition:property}. Then $(n/a_n)_{n \geq 1}$ is non-decreasing and
\begin{equation*}
    \lim_{n \rightarrow +\infty} \frac{a_n}{n} = 0.
\end{equation*}
\end{lemma}

\begin{proof}
We begin by observing that
\begin{equation}\label{an2}
    \frac{a_n}{n} = \prod_{k=1}^{n-1}\frac{k+(2p-1)\alpha_{k+1})}{k+1}
    = \prod_{k=1}^{n-1}\left(1-\frac{1-(2p-1)\alpha_{k+1}}{k+1}\right).
\end{equation}

Let $b_n:=\frac{1-(2p-1)g(\alpha,n+1)}{n+1} \in [0,1)$. Then, $\lim_{n \rightarrow +\infty} \frac{a_n}{n} = 0$ if, and only if, $\sum_{n} b_n=\infty$. Since the DERW satisfies condition \eqref{transition:property}, we have that $\sum_{n} b_n=\infty$. This finishes the proof of the second statement.

Since $(2p-1)\alpha_k\leq1$, it follows from equation \eqref{an2} that
\begin{eqnarray*}
    \frac{n+1}{a_{n+1}}-\frac{n}{a_n}
    &=&\frac{n!(1-(2p-1)\alpha_{n+1})}{\prod_{k=1}^{n}(k+(2p-1)\alpha_{k+1})} \geq 0,
    \end{eqnarray*}
which yields the conclusion of the first statement.    
\end{proof}

Let $M_n$ be the random variable defined by $M_n:=\frac{S_n-\E_x[S_n]}{a_n}$ for $n\geq1$ when $a_n>0$. In subsection \ref{aux.results} we show that the sequence $(M_n)_n$ is a martingale with respect to the natural filtration.

Let $(Y_n)_n$ be the martingale difference sequence associated to $(M_n)_n$, \emph{i.e.}, $Y_1:=M_1$ and $Y_n:=M_n-M_{n-1}$ for $n\geq2$. Since $Y_n = M_n - \E_x[M_n|\F_{n-1}] ~~\p_x-a.s.$, it follows from the linearity of the conditional expectation that
\begin{equation} \label{y_n}
    Y_n = \frac{X_n - \E_x[X_n|\F_{n-1}]}{a_n} \quad \p_x-a.s.
\end{equation}
for all $n \geq 1$. Since we have by \eqref{conditional} that $|\E_x[X_n|\F_{n-1}]| \leq 1 ~~\p_x-a.s.$, one has, for all $n \geq 1$, that
\begin{equation} \label{bounds:y_n}
    |Y_n| \leq \frac{2}{a_n} \quad \p_x-a.s.
\end{equation}
for every $x \in \X$.
    
We now turn to the proof of Theorem \ref{SLLN} and Corollary \ref{limitp2}.

\begin{proof}[Proof of Theorem \ref{SLLN}]
 Set $W_n:=\frac{a_n}{n}Y_n$. By \eqref{y_n}, we verify that $W_n$ is a $\F_n$-measurable random variable such that  $\E_x[W_{n}|\F_{n-1}]=\frac{a_{n}}{n}\E_x[Y_{n}|\F_{n-1}]=0$ $\p_x-a.s.$ Therefore, $(W_n)_n$ is a sequence of bounded martingale differences. 
 
Since $\E_x[W_{n}^2|\F_{n-1}]\leq\frac{4}{n^2} ~~\p_x-a.s.$ by \eqref{bounds:y_n}, we have that, for all $x \in \X$,
\[
    \sum\limits_{j=2}^\infty\E_x[W_j^2|\F_{j-1}]\leq 4\sum\limits_{j=2}^\infty\frac{1}{j^2}<\infty \quad \p_x-a.s.
\]
Thence, it follows from Theorem 2.7 of \cite{hall1980} that $\sum\limits_{j=1}^n\frac{a_j}{j}Y_j=\sum\limits_{j=1}^nW_j$ converges $\p_x$-almost surely as $n \to +\infty$.
    
We apply Lemma \ref{ann} and Kronecker's lemma obtaining that
    \begin{equation*}
        \lim_{n\to+\infty}\left|\frac{S_n - \E_x[S_n]}{n}\right|=\lim_{n\to+\infty}\left|\frac{a_nM_n}{n}\right|=\lim_{n\to+\infty}\left|\frac{\sum_{j=1}^nY_j}{n/a_n}\right|{=}0 \quad\p_x-a.s.
    \end{equation*}
    for all $x \in \X$
\end{proof}
    
To prove Corollary \ref{limitp2} we will use the following result on the expectation of the random walk.

\begin{proposition}\label{media}
    Let $(S_n)_{n\geq 1}$ be the DERW. Then
	\begin{equation}\label{estimatesn}
		\E_x[S_n]=a_n\left(\alpha_1(2q-1)+\sum_{k=1}^{n}\frac{(1-\alpha_k)(2f(T^kx)-1)}{a_k}\right)
	\end{equation}
\end{proposition}
\begin{proof}
Note that $S_1=X_1$ satisfies \eqref{estimatesn} for $n=1$. Suppose that \eqref{estimatesn} holds for a given $n \in \N$. Then
\begin{align*}
\E_x[S_{n+1}] &= \E_x[S_{n}]+\E_x[\E_x[X_{n+1}|\F_{n}]]\\
&=\E_x[S_{n}]\frac{a_{n+1}}{a_{n}}+(1-\alpha_{n+1})(2f(T^{n+1}x)-1)\\
&=a_{n+1}\left(\frac{\E_x[S_{n}]}{a_{n}}+\frac{(1-\alpha_{n+1})(2f(T^{n+1}x)-1)}{a_{n+1}}\right).
\end{align*}
The conclusion follows by induction on $n$. 
\end{proof}
    
\begin{proof}[{Proof of Corollary \ref{limitp2}}]

Since $p=\frac{1}{2}$, $a_n=1$ for all $n\geq1$. Then it can be computed from \eqref{estimatesn} that
\begin{equation*}
\frac{\E_x[S_n]}{n}=\frac{1}{n}\left(\alpha_1(2q-1)+\sum_{k=1}^n(1-\alpha_k)(2f(T^{k}x)-1)\right).
\end{equation*}

Since $\ell(\alpha) = \lim_{n\to+\infty}\alpha_n$ exists, we have that

\begin{align}
\frac{1}{n}\Bigg| \sum_{k=1}^n(1-\alpha_k)(2f(T^{k}x)&-1)-\sum_{k=1}^n(1-\ell(\alpha))(2f(T^{k}x)-1) \Bigg|\nonumber\\ 
&= \frac{1}{n}\left|\sum_{k=1}^n(2f(T^kx)-1)(\ell(\alpha)-\alpha_k)\right| \nonumber \\
&\leq \frac{1}{n}\sum_{k=1}^n\left|\ell(\alpha)-\alpha_k\right|\to0 \quad \text{(as} ~ n \to +\infty\text{)} \label{estimate2}.
\end{align}

It follows from the Birkhoff Ergodic Theorem and the Strong Law of Large Numbers (Theorem \ref{SLLN}) that
\begin{align*}
\lim_{n\to+\infty}\frac{\E_x[S_n]}{n} &= \lim_{n\to+\infty}\frac{1}{n}\left(\alpha_1(2q-1)+\sum_{k=1}^n(1-\ell(\alpha))(2f(T^{k}x)-1)\right)\nonumber\\
&= (1-\ell(\alpha)) \lim_{n\to+\infty} \frac{1}{n}\sum_{k=1}^n(2f(T^{k}x)-1) \\
&= (1-\ell(\alpha))\big(2\E_\mu[f|\mathcal{I}](x)-1\big) \quad \p_x-a.s.
\end{align*}
for $\mu$-almost every $x \in \X$.
\end{proof}


\subsection{The Central Limit Theorem}
Define the sequence of positive terms $(B_n)_{n\geq 1}$ given by $B_n^2:=\sum_{k=1}^n\frac{1}{a_k^2}$.

\begin{lemma}\label{Bnconv}
 If $p \leq 3/4$ or $\ell_{\sup}(\alpha) < \frac{1}{4p-2}$, then $B_n^2 \nearrow +\infty $. Moreover, if $p > 3/4$ and $\ell_{\inf}(\alpha) > \frac{1}{4p-2}$, then $\lim\limits_{n \to \infty}|B_n| < + \infty$.
\end{lemma}

\begin{proof}
    Let $p \leq 3/4$. Then
    \[
        a_n \leq \prod_{j=1}^{n-1}\left( 1 + \frac{1/2}{j}\right) \sim \frac{n^{1/2}}{\Gamma(3
        /2)},
    \]
    which implies that $B_n \nearrow +\infty$. 
    
    Consider now $p > 3/4$ and let $c_n:=\frac{1}{a_n^2}$. Then it follows from Raabe criterion \cite[p.~285]{knopp1954} that $B_n^2\nearrow +\infty$ when $\limsup\limits_{n \to \infty} n \left(1- \frac{c_{n+1}}{c_n}\right)  < 1$, which is satisfied with
    \[
        (4p-2) \ell_{\sup}(\alpha) < 1.
    \]
    In the same fashion, we verify that $\lim\limits_{n\to+\infty}B_n^2 < +\infty$ if $\liminf\limits_{n \to \infty} n \left(1- \frac{c_{n+1}}{c_n}\right)  > 1$. Note that this is satisfied when
    \[
        (4p-2)\ell_{\inf}(\alpha) > 1.
    \]
    \end{proof}
    
    \begin{proof}[{Proof of Theorem \ref{CLT1}}]
    We first examine the asymptotic behaviour of $A_n$. It follows from Lemma \ref{lm:variance:limit} that there exists $n_0 \in \N$ and $c \in (0,1)$ such that, if $n > n_0$, then $c(B_n^2-B_{n_0}^2) \leq A_n^2-A_{n_0}^2$. We apply Lemma \ref{Bnconv} to verify that $A_n \nearrow +\infty$.
    
    Note that we are under conditions of Theorem \ref{SLLN}. Therefore, $\frac{S_n}{n}= \frac{\E_x[S_n]}{n} + \text{o}(1) ~~\p_x-a.s.$ for all $x \in \X$. Then it follows from \eqref{conditional} that $\E_x[X_n|\F_{n-1}]= \E_x[X_n] + \text{o}(1) ~~\p_x-a.s.$ We apply it to \eqref{y_n} obtaining
    \begin{align*}
        \E_x[Y_n^2|\F_{k-1}] & = \frac{1}{a_n^2}\left(1 - 2 \E_x[X_n|\F_{n-1}]\E_x[X_n] +\E_x[X_n]^2 + o(1)\right)
        \\&= \frac{1}{a_n^2}\left(1 - \E_x[X_n]^2 + o(1)\right) \quad \p_x-a.s.
    \end{align*}
    Thus,
    \begin{equation} \label{squared:conv:Yn}
        \frac{\sum_{k=1}^n\E_x[Y_k^2|\F_{k-1}]}{A_n^2} = \frac{A_n^2}{A_n^2} + \frac{B_n^2}{A_n^2}\text{o}(1) =1 + \text{o}(1) \quad \p_x-a.s.,
    \end{equation}
    since $\limsup\limits_{n \to +\infty}|B_n^2/A_n^2| \leq 1/c$.
    
    We may arrive to the desired conclusion applying Corollary 3.1 of \cite{hall1980}. To this end, it suffices to verify that the conditional Lindeberg condition holds. Let   $S_{n,i}$ be the martingale array with martingale difference sequence $X_{n,i}=\frac{Y_{i}}{A_n}$ and $\F_{n,i} := \F_i$.
    
    Consider $p\geq1/2$. Then $a_n\geq1$ and it follows from \eqref{bounds:y_n} that $|X_{n,i}| \leq \frac{2}{A_n} ~\p_x-a.s.$ Recall that, under the given conditions, $A_n \nearrow +\infty$. Hence, for every fixed $\varepsilon>0$, $\p_x\big[|X_{n,i}|>\varepsilon\big]=0$ for sufficiently large $n$.
    
    Let now $p<1/2$. Then $a_j^{-1}\leq a_n^{-1}$ for all $j\leq n$ and it is immediate to see that $|X_{n,i}| \leq \frac{2}{a_nA_n} ~\p_x-a.s.$ We can easily verify that $\lim\limits_{n\to+\infty}a_nA_n=+\infty$ (see proof of Thm. \ref{LIL} for details). Therefore, for all $\varepsilon >0$, one has $\p_x\big[|X_{n,i}|>\varepsilon\big]=0$ for sufficiently large $n$.
    
    Now, fix $\varepsilon>0$. Then
    \begin{equation*}
        \sum_{i=1}^{n}\E_x[X_{n,i}^2\mathbbm{1}_{\{|X_{n,i}|>\varepsilon\}}|\F_{n,i-1}]\rightarrow 0 \quad \p_x-a.s.,
    \end{equation*}
    which completes the proof.
    \end{proof}
    
    \begin{proof}[{Proof of the Corollary \ref{CLTp2}}]
    First, observe that since $p=1/2$, we have that $a_n=1$ for all $n\geq1$. We also may conclude that $\E_x[X_k|\F_{k-1}]=\E_x[X_k] ~~\p_x-a.s$ and $A_n^2=\sum_{k=1}^n(1-\E_x[X_k]^2)=\sum_{k=1}^n(1-(1-\alpha_k)^2(2f(T^kx)-1)^2)$. Consequently,
    \begin{equation}
    \frac{M_n}{A_n}=\frac{\frac{S_n}{\sqrt{n}}-\sqrt{n}\frac{\E_x[S_n]}{n}}{\sqrt{\frac{1}{n}\sum_{k=1}^n(1-(1-g(\alpha,k))^2(2f(T^kx)-1)^2)}}.\label{clt12}
    \end{equation}
    Now, similarly to \eqref{estimate2}, we get
    \begin{equation*}
        \frac{1}{n}\sum_{k=1}^n(1-(1-\alpha_k)^2(2f(T^kx)-1)^2)\sim\frac{1}{n}\sum_{k=1}^n(1-(1-\ell(\alpha))^2(2f(T^kx)-1)^2).
    \end{equation*}
    Note that
    \begin{equation*}
    \begin{split}
        \frac{1}{n}\sum_{k=1}^n(1-(1-\ell(\alpha))^2&(2f(T^kx)-1)^2)\\
        &=1-(1-\ell(\alpha))^2\left(\frac{1}{n}\left[\sum_{k=1}^n4(f(T^kx))^2-4f(T^kx)\right]+1\right).
    \end{split}
    \end{equation*}
We can apply Birkhoff Ergodic Theorem in the last equation to conclude that
    \begin{equation}1-(1-\ell(\alpha))^2(4\E_\mu[f^2|\mathcal{I}]-4\E_\mu[f|\mathcal{I}]+1)\quad\mu-a.s.\label{nAnn}
    \end{equation}
In particular, $\lim\limits_{n\to+\infty}\frac{1}{\sqrt{n}}A_n\neq0$.
    
    The desired conclusion follows applying \eqref{nAnn}, Corollary \ref{limitp2} and Theorem \ref{CLT1} to \eqref{clt12}.
    \end{proof}

\subsection{The Law of Iterated Logarithm}
    \begin{proof}[Proof of Theorem \ref{LIL}]
        The law of iterated logarithm for the DERW follows from a application of Theorems 1 and 2 of \cite{stout1970}. We have already shown \eqref{squared:conv:Yn}.
        
        Define $u_n = \sqrt{2 \log \log A_n^2}$ and $K_n= \frac{2 u_n}{a_n A_n}$. 
        Let us write the inequality \eqref{bounds:y_n} in the following way
        \begin{align*}
            |Y_n| \leq \frac{2}{a_n} = K_n\frac{A_n}{u_n} \quad \mathbb{P}_x-a.s.
        \end{align*}
        It is clear that $K_n$ is $\mathcal{F}_{n-1}$ measurable. In order to get \eqref{log} we need to show that $K_n \to 0 $ as $n \to \infty$.
        
        Recall that $g(\alpha, 2) < 1$. Observe that 
        \begin{align*}
            a_k \geq (1 - g(\alpha, 2))   \prod_{i=2}^{k-1} \left(1 - \frac{1}{i} \right)
            = (1 - g(\alpha, 2)) \frac{1}{k-1},
        \end{align*}
        which implies that 
        \begin{align}\label{An}
            A_n^2 \le \sum_{k = 1}^{n}\frac{1}{a_k^2} \le 
            \frac{1}{\big(1 - g(\alpha, 2)\big)^2} \sum_{k = 1}^{n} (k-1)^2 \le \frac{1}{\big(1 - g(\alpha, 2)\big)^2} n^3.    
        \end{align}
        In other words, $u_n \le \sqrt{2 \log \log \frac{n^3}{1 - g(\alpha, 2)}}$.
        
        Let us now take care of $a_n A_n$. Since there exists $c' >0$ such that $B_n^2/A_n^2$ is bounded by $1/c'$, we get 
        \begin{align*}
            a_n^2A_n^2 & \ge \frac{1}{c'} a_n^2 B_n^2 = \frac{1}{c'} a_n^2 \sum_{j=1}^n \frac{1}{a^2_j} 
            = \frac{1}{c'} \left( 1 + \sum_{j=1}^{n-1} \frac{a_n^2}{a^2_j} \right) \\
            &
            = \frac{1}{c'} \left( 1 + \sum_{j=1}^{n-1} \prod_{i=j}^{n-1} \left(1+\frac{(2p -1)g(\alpha, i+1)}{i}\right)^2 \right),
        \end{align*}
        which yields 
        \begin{align*}
            \prod_{i=j}^{n-1} \left(1+\frac{(2p -1)g(\alpha, i+1)}{i}\right) \ge \prod_{i=j}^{n-1} \left(1 - \frac{1}{i}\right)
            = \frac{j-1}{n-1}.
        \end{align*}
        Thence, 
        \begin{align}\label{anAn}
            a_n^2 A_n^2 \ge \frac{1}{c'} \left( 1 + \sum_{j=1}^{n-1} \left(\frac{j-1}{n-1}\right)^2 \right) \ge  \frac{n}{6 \, c'}.
        \end{align}
        Combining \eqref{anAn} and \eqref{An} we are able to show that 
        \begin{align*}
            K_n = \frac{2 u_n}{a_n A_n} \le 4\sqrt{\frac{3 \, c' \log \log \left( \frac{n^3}{(1 - g(\alpha, 2))^2}\right)}{ n }} \to 0  \quad \mbox{ as } \, n \to \infty,
        \end{align*}
        and the claim \eqref{log} follows.
    \end{proof}

    \begin{proof}[{Proof of Theorem \ref{ASC}}]
    We first observe that $\lim\limits_{n \to +\infty} B_n <+\infty$ by Lemma \ref{Bnconv}.
    
    It follows from \eqref{bounds:y_n} that $\sum_{k=1}^n\E_x[Y_k^2]\leq4B_n^2 ~~\p_x-a.s.$. Then by Theorem 12.1 of \cite{williams1991}, for all $x \in \X$,
    \[
    M_n=\sum_{k=1}^nY_k=\frac{S_n-\E_x[S_n]}{a_n}\to M \quad\p_x-a.s ~~ \text{and in} ~~ \mathcal{L}^2.
    \]
    Since $\E_x[M_n]=0$ for all $n$, $|\E_x[M]|=|\E_x[M-M_n]|\leq\E_x[|M-M_n|]$. Then
    \begin{equation*}
    |\E_x[M]|\leq\E_x[|M-M_n|^2]^{1/2}\to0\quad\text{as }n\to+\infty.
    \end{equation*}
    
    Furthermore, since $(Y_n)_{n\geq1}$ is a bounded martingale difference sequence in $\mathcal{L}^2$ and it converges almost surely,
    \begin{equation*}
    \Var_x[M]=\lim_{n\to+\infty}\Var_x[M_n]=\sum_{k=1}^\infty\E_x[Y_k^2]>0.
    \end{equation*}
    This leads us to conclude that $M$ is a non-degenerated zero mean random variable.
    \end{proof}


\subsection{Auxiliary Results} \label{aux.results}
\begin{proposition}
The sequence of random variables $(M_n)_{n\geq1}$ defines a zero mean martingale.
\end{proposition}
\begin{proof}
The zero mean property being straightforward, we only prove that $\E_x[M_{n+1} \\ |\F_n]= M_n$. Indeed, it follows from equation \eqref{conditional} that
\begin{eqnarray*}
\E_x[M_{n+1}|\F_n]&=&\frac{S_n-\E_x[S_n]}{a_{n+1}}+\frac{\E_x[X_{n+1}|\F_n]-\E_x[X_{n+1}]}{a_{n+1}}\\
    &=&\frac{1}{a_{n+1}}\left(S_n\left(1+\frac{\alpha_n(2p-1)}{n}\right)-\left(1+\frac{\alpha_n(2p-1)}{n}\right)\E_x[S_n]\right)\\
    &=&\frac{S_n-\E_x[S_n]}{a_n} \quad \p_x-a.s.
\end{eqnarray*}
which finishes the proof.
\end{proof}

\begin{lemma} \label{lm:variance:limit}
    Let the DERW be defined with $p<1$ and $\ell_{\inf}(\alpha)>0$. Then
    \begin{equation} \label{var:lim:less:than:one}
        \liminf_{n\to+\infty}\Var_x[X_n] >0.
    \end{equation}
    Futhermore, if the DERW jointly satisfies \eqref{transition:property}, \eqref{D1}, and \eqref{D2}, then \eqref{var:lim:less:than:one} still holds.
\end{lemma}
\begin{proof}
    Note that \eqref{var:lim:less:than:one} is equivalent to $\limsup\limits_{n\to+\infty}|\E_x[X_n]|<1$. Since $\frac{S_{n-1}}{n-1} \in [-1,1]$, it follows from \eqref{conditional} that conditions
    \begin{eqnarray}
        \liminf_{n\to+\infty}\big(\alpha_n(1-p) + (1-\alpha_{n})f(T^{n}(x))\big) >0 \label{lim:minimal}\\
        \limsup\limits_{n\to+\infty}\big(\alpha_{n}p+(1-\alpha_{n})f(T^{n}(x))\big) <1 \label{lim:maximal}
    \end{eqnarray}
    are sufficient to verify \eqref{var:lim:less:than:one}. Moreover,
    \begin{eqnarray}
        (1-p)\ell_{\inf}(\alpha) + \big(1-\ell_{\sup}(\alpha)\big)\liminf\limits_{n\to+\infty}f\big(T^{n}(x)\big) >0 \label{lim:minimal:2}\\
        p\cdot\ell_{\sup}(\alpha)+\big(1-\ell_{\inf}(\alpha)\big)\limsup\limits_{n\to+\infty}f\big(T^{n}(x)\big) <1 \label{lim:maximal:2}
    \end{eqnarray}
    satisfy \eqref{lim:minimal} and \eqref{lim:maximal}, respectively.
    
    Let $p<1$ and $\ell_{\inf}(\alpha)>0$. Then \eqref{lim:minimal:2} is immediately satisfied. Since $f(T^n(x)) \leq 1$, we verify \eqref{lim:maximal} by noting that
    \[
        \limsup\limits_{n\to+\infty}\left(\alpha_{n}p+(1-\alpha_{n})f(T^{n}(x))\right) \leq 1-(1-p)\ell_{\inf}(\alpha_n) <1.
    \]
    
    Now, if $p=1$ or $\ell_{\inf}(\alpha)=0$, then \eqref{lim:minimal:2} and \eqref{lim:maximal:2} are straightforwardly satisfied since conditions \eqref{transition:property}, \eqref{D1}, and \eqref{D2} hold.
\end{proof}

We finish this section by proving Corollary \ref{cor:D}.

\begin{proof}[{Proof of Corollary \ref{cor:D}}]
        The proof follows in the same lines as those of the proof of Theorem \ref{CLT1} and it is an immediate consequence of Lemma \ref{lm:variance:limit} and Theorem \ref{SLLN}.
    \end{proof}


\section*{Acknowledgements}
C.F.C. thanks FAPESP (grant \#2017/10555-0), L.R.L. thanks FAPESP (grant \#2019/19056-2), and R.J.G. thanks FAPESP (grants \#2017/10555-0 and\break\#2018/04764-9) for financial support. This study was financed in part by the Coordena\c{c}\~ao de Aperfei\c{c}oamento de Pessoal de N\'ivel Superior - Brasil (CAPES) - Finance Code 001.

\small
\bibliographystyle{abbrv}  
\bibliography{references}

\end{document}